\documentclass{amsart}
\usepackage{amsxtra}
\usepackage{ stmaryrd }
\usepackage{amssymb,amsmath,amsthm,stmaryrd}
\usepackage{color}
\usepackage[all]{xy}
\usepackage{cancel}
\theoremstyle{plain}

\newtheorem{theorem}{Theorem}[section]
\newtheorem{lemma}[theorem]{Lemma}
\newtheorem{proposition}[theorem]{Proposition}
\newtheorem{corollary}[theorem]{Corollary}
\newtheorem{definition}[theorem]{Definition} \theoremstyle{definition}
\newtheorem{example}[theorem]{Example}
\newtheorem{remark}[theorem]{Remark}

\newcommand{\rr}       {\rightrightarrows}

\newcommand{\mx}{\mathfrak{X}} \newcommand{\dr}{\mathbf{d}}
 \newcommand{\ldr}[1]{{{\pounds}}_{#1}}

\newcommand{\an}[1]{\arrowvert_{#1}}

 \DeclareMathOperator{\pr}{pr}
\DeclareMathOperator{\Hom}{Hom}
\DeclareMathOperator{\End}{End}

\DeclareMathOperator{\id}{id}

\begin{document}

%%%%%%%%%%%%%%%%%%%%%%%%%%%%%%%%%%%%%%%%%%%%%%%%%%%%%%%%%%%%%%%%%%%%%%%%%%%
%%%%%%%%%%%%%%%%%%%%%%    Title    %%%%%%%%%%%%%%%%%%%%%%%%%%%%%%%%%%%%%%%%
\title{Infinitesimal ideal systems and the Atiyah class}

%%% author one information

\author{M. Jotz Lean}
\address{Mathematisches Institut, Georg-August-Universit\"at G\"ottingen \\
\texttt{madeleine.jotz-lean@mathematik.uni-goettingen.de}}

\begin{abstract}
  This short note gives a geometric interpretation of the Atiyah class
  of a Lie pair. It proves that it vanishes if the
  subalgebroid is the kernel of a fibration of Lie algebroids. In other
  words, the Atiyah class of a Lie pair vanishes if the subalgebroid
  is the fiber of an ideal system in the Lie algebroid. In order to
  prove this, a new characterisation of the ideal condition for an
  infinitesimal ideal system is found.


\end{abstract}

\maketitle

\tableofcontents

\section{Introduction}
Let $A\to M$ be a Lie algebroid and $J\subseteq A$ a subalgebroid.
The pair $(A,J)$ is called a \emph{Lie pair} \cite{ChStXu16}. Lie
pairs arise e.g.~in Lie theory, complex geometry and foliation
theory. For instance, a smooth foliation $F$ in the tangent algebroid
$TM$ of a smooth manifold $M$ defines a Lie pair $(TM,F)$.

Given a Lie pair, the quotient $A/J$ always carries a flat
representation of $J$, the \emph{Bott connection} \cite{Bott68,Bott72}
\[\nabla^J\colon \Gamma(J)\times\Gamma(A/J)\to\Gamma(A/J), \quad
  \nabla^J_j\bar a=\overline{[j,a]}
\]
for $j\in\Gamma(J)$ and $a\in\Gamma(A)$.  An \emph{extension} of the
Bott connection $\nabla^J$ is a linear $A$-connection on $A$, that
preserves $J$ (in the second argument) and induces so a quotient
connection $\Gamma(A)\times\Gamma(A/J)\to\Gamma(A/J)$ that restricts
to $\nabla^J$ in the first argument.  The \emph{Atiyah class}
\cite{ChStXu16} of the Lie pair $(A,J)$ is a cohomology class
$\alpha_J\in H^1(J,\Hom(A/J,\End(A/J)))$ that vanishes if and only if
$\nabla^J$ has a extension $\nabla$ that is \emph{transverse} to it:
For $a, b\in \Gamma(A)$ two $\nabla^J$-flat sections, $\nabla_ab$ is
again $\nabla^J$-flat -- see also \cite{Molino71a,Molino71b}
for the more classical setting of foliations.  However, from a more
elementary geometric point of view, it is not completely clear what
the existence of a transverse connection means for a Lie pair.  In the
case of a Lie pair $(TM,F)$ as above, the Atiyah class vanishes for
instance if $F$ is simple, i.e.~ if the space of leaves $M/F$ is a
smooth manifold. A transverse connection is in this case a projectable
connection -- such a connection always exists for a surjective
submersion $M\to M'$ with connected fibers.

This short note generalises this fact to give a geometric
interpretation of the Atiyah class of a Lie pair $(A,J)$. More
precisely, if there exists a fibration of Lie algebroids $A\to A'$
over a smooth surjective submersion $M\to M'$ with connected fibers,
such that $J$ is the kernel of the fibration, then its Atiyah class
vanishes.

\medskip

Such a fibration of Lie algebroids is always encoded in an
infinitesimal ideal structure on the Lie pair $(A,J)$. Such an
\emph{infinitesimal ideal system} is a triple $(F_M,J,\nabla)$, where
$F_M\subseteq TM$ is an involutive subbundle such that
$\rho(J)\subseteq F_M$ and $\nabla$ is a flat $F_M$-connection on
$A/J$ with the following properties:
\begin{enumerate}
\item[(iis1)] If $a\in\Gamma(A)$ is $\nabla$-parallel, then $[a,j]\in\Gamma(J)$ 
for all $j\in\Gamma(J)$.
\item[(iis2)] If $a,b\in\Gamma(A)$ are $\nabla$-parallel, then $[a,b]$ is also $\nabla$-parallel.
\item[(iis3)] If $a\in\Gamma(A)$ is  $\nabla$-parallel, then $\rho(a)$ is
  $\nabla^{F_M}$-parallel, where\linebreak
  $\nabla^{F_M}\colon\Gamma(F_M)\times\Gamma(TM/F_M)\to\Gamma(TM/F_M)$
  is the Bott connection.
\end{enumerate}
The definition is due to \cite{JoOr14}, but the structure already
appeared in \cite{Hawkins08} in geometric quantisation as the
infinitesimal version of polarisations on groupoids, and the case
$F_M=TM$ has been studied independently in \cite{CrSaSt12} in relation
with a modern approach to Cartan's work on pseudogroups.

For simplicity, \emph{ideal} is here short for \emph{infinitesimal
  ideal system}.  Ideals were called ``IM-foliations'' in an early
version of \cite{JoOr14} and in \cite{ZaZh12}, in analogy to the
``IM-2-forms'' of \cite{BuCrWeZh04}.  The new terminology is more
adequate, since an infinitesimal ideal system is an infinitesimal
version of the \emph{ideal systems} in \cite{Mackenzie05, HiMa90b}.
Let $A$ be a Lie algebroid over $M$. An ideal system of $A$ is a
triple $(R,J,\theta)$ where $J\subseteq A$ is a Lie subalgebroid,
$R = R(f)$ is a closed, embedded, wide, Lie subgroupoid of the pair
groupoid $M\times M\rr M$ corresponding to a surjective submersion
$f\colon M\to \bar M$:
\[ R(f)=\{(p,q)\in M\times M\mid f(p)=f(q)\},\]
and where $\theta$ is a linear action of $R(f)$ on
the vector bundle $A/J\to M$, such that:
\begin{enumerate}
  \item $\rho(J)\subseteq T^fM$;
\item  if $a,b\in\Gamma(A)$ are
  $\theta$-stable;
  \[ \theta((p,q), \bar a(q))=\bar a(p) \text{ for all } (p,q)\in R
  \]
  (and similarly for $b$),
  then $[a, b]$ is $\theta$-stable;
\item  if $a\in\Gamma(A)$ is
  $\theta$-stable and  $j\in\Gamma(J)$, then $[a,j]\in\Gamma(J)$;
\item the induced map $\bar\rho\colon A/J \to TM/T^fM$ is
  $R(f)$-equivariant with respect to $\theta$ and the canonical action
  $\theta_0$ of $R(f)$ on $TM/T^fM$.
\end{enumerate}
Here, the canonical action $\theta_0$ is the action
transported from the pullback bundle by the isomorphism $TM/T^fM\simeq 
f!T\bar M$ .

An infinitesimal ideal system $(F_M, J, \nabla)$ with $F_M$ simple and
$\nabla$ with trivial holonomy always integrates to an ideal system:
$f\colon M\to M/F_M$ is the surjective submersion, so $F_M= T^fM$, and
$\theta$ is given by parallel transport along the leaves of $F_M$, see
\cite{JoOr14}. The conditions on $F_M$ and $\nabla$ are equivalent to
the quotient $(A/J)/\nabla \simeq (A/J)/\theta\to M/F_M$ being a
vector bundle, which inherits then a Lie algebroid structure `for
free' such that the quotient map is a fibration of Lie algebroids
\cite{JoOr14}. Ideal systems were in fact defined as the kernels of
fibrations of Lie algebroids. If such an infinitesimal ideal system
$(F_M,J, \nabla)$ exists, then \emph{$A$ is reducible by the fiber
  $J$}, since $J$ becomes the kernel of a fibration of Lie algebroids.

In general, a Lie pair will not carry an infinitesimal ideal
structure, so the Lie pairs $(A,J)$ which do must be considered
special. The author's general goal is to find obstructions to
Lie pairs carrying infinitesimal ideal system structures.

\medskip

This paper solves a slightly different problem and proves that if a
Lie algebroid $A$ is reducible by the fiber $J$, then the Atiyah class
of the Lie pair $(A,J)$ is zero.
\begin{theorem}\label{main}
  Let $A$ be a Lie algebroid on $M$ and $J$ a subalgebroid. If
  there exists an infinitesimal ideal system $(F_M,J,\nabla^i)$ in
  $A$, such that the quotient $(A/J)/\nabla^i\to M/F_M$ exists and is smooth, then
  the Atiyah class of the Lie pair $(A,J)$ vanishes.
\end{theorem}

Since the infinitesimal ideal systems $(F_M,J,\nabla^i)$ such that the
quotient Lie algebroid $(A/J)/\nabla^i\to M/F_M$ exists are exactly
the infinitesimal ideal systems that integrate to ideal systems
\cite{JoOr14}, the theorem can be reformulated as follows.

\begin{theorem}\label{main'}
  Let $A$ be a Lie algebroid on $M$ and $J$ a subalgebroid. If
  there exists an ideal system $(R,J,\theta)$ in $A$, then the Atiyah
  class of the Lie pair $(A,J)$ vanishes.
\end{theorem}

\subsection*{Outline of the paper}
Section \ref{sec_fibrations} considers fibrations of vector bundles
and infinitesimal ideal systems in vector bundles (i.e.~Abelian Lie
algebroids). It shows that the Atiyah class of such an infinitesimal
ideal system vanishes if the infinitesimal ideal system defines a
fibration of the vector bundle.  Section \ref{prelim_iis} proves a new
formulation of the ideal condition (iis1) in the definition of an
infinitesimal ideal system. It further explains fibrations of Lie
algebroids versus infinitesimal ideal systems \cite{JoOr14}, and
defines the Atiyah class of an infinitesimal ideal system.  Section
\ref{sec_proof_main} proves the main theorem (Theorem \ref{main}).

\subsection*{Acknowledgements}
The author thanks Mathieu Sti\'enon for inspiring conversations at the
origin of this paper.

\section{Preliminaries}\label{prelim}

\subsection{Flat connections and fibrations of vector bundles}\label{sec_fibrations}
Consider a fibration of vector bundles 
\begin{equation*}
\begin{xy}
\xymatrix{
E\ar[r]^{\varphi}\ar[d]_{q_E}& E'\ar[d]^{q_{E'}}\\
M\ar[r]_{\varphi_0}&M'
}
\end{xy}
\end{equation*}
That is, the map $\varphi_0$ is a surjective submersion with connected
fibres and $\varphi^!\colon E\to \varphi_0^!E'$ is a surjective vector
bundle morphism over the identity on $M$.

\begin{proposition}\label{existence_of_connection}
  Let $\varphi\colon E\to E'$ be a fibration of vector bundles over
  $\varphi_0\colon M\to M'$.  For any linear connection
  $\nabla'\colon \mx(M')\times\Gamma(E')\to\Gamma(E')$ there exists a
  linear connection $\nabla\colon \mx(M)\times\Gamma(E)\to\Gamma(E)$
  such that if $X\sim_{\varphi_0}X'$ and $e\sim_\varphi e'$ then
  $\nabla_Xe\sim_\varphi \nabla'_{X'}e'$.
\end{proposition}

\begin{proof} Let $K\subseteq E$ be the kernel of the map $\varphi$;
  $K$ is a subbundle of $E$. The inclusion of $K$ in $E$ is denoted by
  $\iota$. Set $F_M=T^{\varphi_0}M$, an involutive and simple
  subbundle of $TM$.  Choose a splitting $j\colon \varphi_0^!E'\to E$
  of the short exact sequence $0\to K\to E\to \varphi_0^!E'\to 0$ of
  vector bundles over $M$. Recall that each smooth section $e'$ of
  $E'\to M'$ defines a smooth section $e'_{!}$ of $\varphi_0^!E'\to M$
  by $e'_{!}(m)=(m,e'(\varphi_0(m)))$ for all $m\in M$, and that the
  space of sections of $\varphi_0^!E'$ is generated as a
  $C^\infty(M)$-module by these sections.  Further,
  $\varphi\circ j\circ e'_{!}=e'\circ\varphi_0$, or in other words
  $(j\circ e'_{!})\sim_\varphi e'$.

  The connection $\nabla'$ defines as follows a connection
  $\varphi^!\nabla'\colon\mx(M)\times\Gamma(\varphi_0^!E')\to\Gamma(\varphi_0^!E')$:
  since $\nabla^{F_M}$ is flat, it is sufficient to define
  $\varphi^!\nabla'$ on vector fields $X\in\mx(M)$ that are
  $\varphi_0$-projectable to vector fields on $M'$, i.e.~such that the
  class of $X$ in $\Gamma(TM/F_M)$ is $\nabla^{F_M}$-flat. Set
\[(\varphi^!\nabla')_Xe'_{!}=(\nabla'_{X'}e')_{!}
\]
where $X'\in\mx(M')$ is the projection of $X$;
$X\sim_{\varphi_0}X'$. In particular, $(\varphi^!\nabla')_Xe'_{!}=0$ for $X\in\Gamma(F_M)$.

Choose any connection $\tilde\nabla\colon\mx(M)\times\Gamma(K)\to\Gamma(K)$ and 
define a connection $\nabla\colon\mx(M)\times\Gamma(E)\to\Gamma(E)$ by
\[\nabla_X(\iota(k)+j(e'))=\iota(\tilde\nabla_Xk)+j((\varphi^!\nabla')_Xe')
\]
for all $e'\in\Gamma(\varphi_0^!E')$ and all $k\in\Gamma(K)$. It is
now easy to check that if $X\sim_{\varphi_0}X'$ and $e\sim_\varphi e'$
then $\nabla_Xe\sim_\varphi \nabla'_{X'}e'$.
\end{proof}

Consider here also a fibration of vector bundles
$\varphi\colon E\to E'$ over $\varphi_0\colon M \to M'$ and set
$K:=\ker(\varphi)\subseteq E$ and $F_M:=T^{\varphi_0}M\subseteq TM$,
an involutive subbundle.  Here and in the following, $\bar e$ always
denotes the projection of $e\in\Gamma(E)$ to a section of
$\Gamma(E/K)$.  Define a flat connection
$\nabla^\varphi\colon \Gamma(F_M)\times \Gamma(E/K)\to\Gamma(E/K)$ by
setting $\nabla^\varphi_X\bar e=0$ for all sections $e\in \Gamma(E)$
that are $\varphi$-related to some section $e'\in\Gamma(E')$,
i.e.~such that $\varphi\circ e=e'\circ \varphi_0$; see \cite{JoOr14}.

Assume that a linear connection
$\nabla\colon\mx(M)\times\Gamma(E)\to\Gamma(E)$ projects to a
connection $\nabla'\colon\mx(M')\times\Gamma(E')\to\Gamma(E')$. Then
for $X\sim_{\varphi_0} X'$ and $e\sim_\varphi e'$, the sections
$\nabla_Xe$ and $\nabla'_{X'}e'$ are $\varphi$-related;
$\nabla_Xe\sim_\varphi \nabla'_{X'}e'$. In other words, if $\bar X$ is
a flat vector field for the $F_M$-Bott connection on $TM/F_M$ and
$\bar e$ is $\nabla^\varphi$-flat, then $\overline{\nabla_Xe}$ must be
$\nabla^\varphi$-flat and project to $\nabla'_{X'}e'$.

In particular if $X\sim_{\varphi_0} X'$ and $e\sim_\varphi 0$,
i.e. $e\in\Gamma(K)$, then again $\nabla_Xe$ must be a section of
$K$. This shows that $\nabla$ restricts to a connection
$\nabla\colon\mx(M)\times\Gamma(K)\to \Gamma(K)$ and defines so a
quotient connection
$\bar\nabla\colon\mx(M)\times\Gamma(E/K)\to\Gamma(E/K)$.

Now if $X\in\Gamma(F_M)$, then $X\sim_{\varphi_0} 0$ and so
$\nabla_Xe\sim_\varphi 0$ for all $\nabla^\varphi$-flat\footnote{By
  abuse of notation, say that $e\in\Gamma(E)$ is
  $\nabla^\varphi$-flat if $\bar e\in\Gamma(E/K)$ is
  $\nabla^\varphi$-flat.}  sections $e\in\Gamma(E)$. This means that
$\nabla_Xe\in\Gamma(K)$ for those sections. Then
$\bar \nabla_X\bar e=0=\nabla^\varphi_X \bar e$, and since the
$\nabla^\varphi$-flat sections of $E/K$ generate all sections of $E/K$
as a $C^\infty(M)$-module, this implies that the restriction of
$\bar\nabla$ to $\Gamma(F_M)\times\Gamma(E/K)\to\Gamma(E/K)$ equals
the connection $\nabla^\varphi$. This yields the following
proposition.

\begin{proposition}\label{prop2.2}
  Let $\varphi\colon E\to E'$ a fibration of vector bundles over
  $\varphi_0\colon M\to M'$ and consider $\nabla^\varphi$ as above.
  Take a linear connection
  $\nabla\colon \mx(M)\times\Gamma(E)\to\Gamma(E)$ that projects to a
  linear connection
  $\nabla'\colon \mx(M')\times\Gamma(E')\to\Gamma(E')$ as in
  Proposition \ref{existence_of_connection}. Then $\nabla$ is an
  extension of $\nabla^\varphi$ in the sense of the following
  definition.
\end{proposition}

\begin{definition}\label{def_nabla_p}
Let $E\to M$ be a vector bundle, $K\subseteq E$ a subbundle and
$F_M\subseteq TM$ an involutive subbundle, with a flat connection 
\[\nabla^p\colon \Gamma(F_M)\times\Gamma(E/K)\to\Gamma(E/K).\]
Then a connection
$\nabla\colon\mx(M)\times\Gamma(E)\to\Gamma(E)$ is an \emph{extension
of $\nabla^p$} if 
\begin{enumerate}
\item $\nabla$ restricts to a connection
  $\mx(M)\times\Gamma(K)\to\Gamma(K)$ and 
\item the induced (quotient) connection
  $\bar\nabla\colon\mx(M)\times\Gamma(E/K)\to\Gamma(E/K)$ satisfies 
$\bar \nabla_Y=\nabla^p_Y$ for all $Y\in\Gamma(F_M)$.
\end{enumerate}
\end{definition}

 In general, $\nabla^p$ as in Definition \ref{def_nabla_p} does not
  need to come from a vector bundle fibration. If $F_M$ is simple
  and $\nabla^p$ has no holonomy, then it defines a fibration of vector bundles \cite{JoOr14}:
\begin{equation*}
\begin{xy}
\xymatrix{
E\ar[r]\ar[d]_{q_E}& (E/K)/\nabla^p\ar[d]^{[q_E]}\\
M\ar[r]&M/F_M
}
\end{xy}
\end{equation*}
The space $(E/K)/\nabla^p$ is the quotient of $E/K$ by the equivalence
relation generated by parallel transport along paths in the leaves of
$F_M$. That is, $\bar e_m$ and $\bar e'_{m'}$ are equivalent if there
exists a $\nabla^p$-flat section $\bar e$ of $E/K$ with
$\bar e(m)=\bar e_m$ and $\bar e(m')=\bar e'_{m'}$.  The triple
$(F_M,K,\nabla^p)$ is in fact an infinitesimal ideal system in
the (Abelian) Lie algebroid $(E\to M, \rho=0, [\cdot\,,\cdot]=0)$.

The upper index $p$ in Definition \ref{def_nabla_p} means to recall
that in general, $\nabla^p$ is what can be seen as an infinitesimal
vector bundle ``projection'' -- i.e.~the right candidate for defining
a fibration of vector bundles -- even if $M/F_M$ is not a smooth
manifold, or $\nabla^p$ has non trivial holonomy.  The remainder of
this section shows that if $\nabla^p$ is defined by a fibration, then
its \emph{Atiyah class} must vanish. The first step towards the
construction of the Atiyah class of $\nabla^p$ is the choice of an
extension of $\nabla^p$. This is possible due to the following lemma.

\begin{lemma}\label{ex_of_ext}
  Let $F_M\subseteq TM$ be an involutive subbundle of the tangent of a
  smooth manifold $M$. Let $E\to M$ be a smooth subbundle, and
  $K\subseteq E$ a subbundle.  There always exists an extension of a
  linear connection
  \[\nabla^p\colon\Gamma(F_M)\times\Gamma(E/K)\to\Gamma(E/K).
    \]
    Two extensions differ by a form
    $\phi\in \Gamma(T^*M\otimes E^*\otimes E)$ that
    induces a well-defined form
    $\bar\phi\in\Gamma(\Hom(TM/F_M,\End(E/K)))$ by
    $\bar\phi(\bar X)(\bar e)=\overline{\phi(X,e)}$ for all
    $X\in\mx(M)$, $e\in\Gamma(E)$.
  \end{lemma}

\begin{proof}
  Choose any $TM$-connection $\nabla^K$ on $K$. Choose any smooth
  complement $H$ of $F_M$ in $TM$ and an arbitrary connection
  $\nabla^H\colon\Gamma(H)\times\Gamma(E/K)\to\Gamma(E/K)$.  Then define
  $\bar\nabla\colon\mx(M)\times\Gamma(E/K)\to\Gamma(E/K)$ by
  \[ \bar\nabla_{X+Y}\bar e=\nabla^p_X\bar e+\nabla^H_Y\bar e
  \]
  for $X\in\Gamma(F_M)$, $Y\in \Gamma(H)$ and $\bar e\in\Gamma(E/K)$.
  A choice of splitting $\sigma\colon E/K\to E$ of the short exact sequence
  \[ 0\rightarrow K\rightarrow E\overset{\pr}{\rightarrow} E/K\rightarrow 0
  \]
  gives an isomorphism $E\simeq K\oplus E/K$,
  $E\ni e\simeq (e-\sigma(\bar e)) +\bar e\in K\oplus E/K$.  Set
  $\nabla\colon\mx(M)\times\Gamma(E)\to\Gamma(E)$,
  $\nabla_Xe=\nabla^K_X(e-\sigma(\bar
  e))+\sigma(\bar\nabla_X\bar e)$ for $X\in\mx(M)$ and
  $e\in\Gamma(E)$. The linear connection $\nabla$ is an
  extension of $\nabla^p$.

  \medskip

  Assume that $\nabla'\colon\mx(M)\times\Gamma(E)\to\Gamma(E)$ is a
  second extension of $\nabla^p$. Then as usual
  $\nabla-\nabla'=:\phi$ is a section of
  $T^*M\otimes E^*\otimes E$.  By definition of an extension,
  $\phi(X,k)=\nabla_Xk-\nabla'_Xk\in\Gamma(K)$ for all
  $X\in\mx(M)$ and $k\in\Gamma(K)$. Hence there is an induced
  $\tilde\phi\in\Gamma(T^*M\otimes(E/K)^*\otimes E/K)$,
  $\tilde \phi(X,\bar
  e)=\overline{\phi(X,e)}=\overline{\nabla_Xe-\nabla'_Xe}$ for
  all $X\in\mx(M)$ and $e\in\Gamma(E)$. Now
  $\overline{\nabla_Xe-\nabla'_Xe}=\nabla^p_X\bar
  e-\nabla^p_X\bar e=0$ for all $X\in\Gamma(F_M)$ and $e\in\Gamma(E)$
  yields the existence of $\bar\phi$ as in the statement.
  \end{proof}

\bigskip

Return to the situation of Proposition \ref{existence_of_connection}.
Choose again $X\in \mx(M)$ a $\nabla^{F_M}$-flat vector field, i.e.~a
vector field such that $X\sim_{\varphi_0}X'$ for some $X'\in\mx(M')$;
see \cite{JoOr14}.  Choose also $Y$ a section of $F_M$ and
$\bar e\in\Gamma(E/K)$ a $\nabla^\varphi$-flat section, i.e.~such that
$e\sim_\varphi e'$ for some $e'\in\Gamma(E')$.

Then on the one hand $[Y,X]\in\Gamma(F_M)$ by definition of
the Bott connection, and so
\[
\bar\nabla_X\nabla_Y^\varphi\bar e+\bar\nabla_{[Y,X]}\bar e
=\bar\nabla_X(0)+\nabla^\varphi_{[Y,X]}\bar e=0.\]
On the other hand, 
\[\nabla^\varphi_Y(\bar\nabla_X\bar
  e)=\nabla^\varphi_Y\overline{\nabla_Xe}=0\] since $\nabla$ projects under $\varphi$
to $\nabla'$, so $\nabla_Xe\sim_\varphi\nabla'_{X'}e'$ and
$\overline{\nabla_Xe}$ is therefore necessarily $\nabla^\varphi$-flat.
Hence, the condition that $\nabla$ projects to a connection $\nabla'$
reads as follows.
\begin{lemma}\label{cond_for_proj}
  Let $\varphi\colon E\to E'$ be a fibration of vector bundles over
  $\varphi_0\colon M\to M'$.  If a connection
  $\nabla\colon \mx(M)\times\Gamma(E)\to\Gamma(E)$ projects to a
  connection $\nabla'\colon \mx(M')\times\Gamma(E')\to\Gamma(E')$, then 
  $\nabla$ is an extension of $\nabla^\varphi$ and 
\begin{equation} \label{alpha_def}
\nabla^\varphi_Y(\bar\nabla_X\bar
e) -\bar\nabla_X\nabla_Y^\varphi\bar e-\bar\nabla_{[Y,X]}\bar e=0
\end{equation}
for all $\nabla^{F_M}$-flat vector fields $X\in \mx(M)$, for all
$Y\in\Gamma(F_M)$ and all $\nabla^\varphi$-flat sections $\bar
e\in\Gamma(E/K)$.
\end{lemma}
  Write $\omega(Y,X,\bar e)$ for the left-hand side
of this equation.  An easy computation shows that $\omega(Y,fX,\bar
e)=f\omega(Y,X,\bar e)$ for all $f\in C^\infty(M)$.  Since
$\nabla^{F_M}$-flat vector fields span $\mx(M)$ as
$C^\infty(M)$-module, the section $\omega(Y,X,\bar e)$
must consequently be $0$ for \emph{all} $X\in\mx(M)$. In the same manner, 
then $\omega(Y,X,f\bar e) =f\omega(Y,X,\bar e)$ and 
\eqref{alpha_def} must actually be true for all $Y\in\Gamma(F_M)$,
$X\in\mx(M)$ and $e\in\Gamma(E)$. Furthermore, $\omega(Y_1,Y_2,\bar
e)=R_{\nabla^\varphi}(Y_1,Y_2)\bar e=0$ for $Y_1,Y_2\in\Gamma(F_M)$
shows that $\omega$ can be seen as an element $\omega_{\nabla}$ of
$\Omega^1(F_M, \operatorname{Hom}(TM/F_M, \operatorname{End}(E/K)))$:
\begin{equation}\label{def_at_form}
  \omega_\nabla(Y)(\bar X)\bar e=\overline{R_\nabla(Y,X)e}.
  \end{equation}

\medskip

Note that $\nabla^{F_M}$ and $\nabla^\varphi$ (or $\nabla^p$) induce a
flat $F_M$-connection on
$\operatorname{Hom}(TM/F_M, \operatorname{End}(E/K))$, which is given
by
\[(\nabla^{\operatorname{Hom}}_Y\phi)(\bar X)(\bar e)=\nabla^\varphi_Y(\phi(\bar X)(\bar e))
-\phi(\nabla_Y^{F_M}\bar X)(\bar e)-\phi(\bar X)(\nabla_Y^{\varphi}\bar e).
\] 
For simplicity, the induced differential operator
$\dr_{\nabla^{\operatorname{Hom}}}$ on $\Omega^\bullet(F_M, \operatorname{Hom}(TM/F_M,
\operatorname{End}(E/K)))$ is denoted by $\dr_{\nabla^p}$.

\begin{proposition} \label{d_vanishes} Let $\nabla^p$ be a flat
  $F_M$-connection on $E/K$ and choose an extension
  $\nabla\colon\mx(M)\times\Gamma(E)\to\Gamma(E)$  of
  $\nabla^p$. Define $\omega_\nabla$ as in \eqref{def_at_form}. Then
\begin{equation}\label{closed}
\dr_{\nabla^p}\omega_{\nabla}=0
\end{equation}
and the class of $\omega_{\nabla}$ in
$H^1_{\dr_{\nabla^p}}(F_M,
\operatorname{Hom}(TM/F_M, \operatorname{End}(E/K)))$ does not depend
on the choice of the extension $\nabla$. 
\end{proposition}

\begin{proof}
For $Y_1, Y_2\in\Gamma(F_M)$ the value of
$(\dr_{\nabla^p}\omega_{\nabla})(Y_1,Y_2)$ at $\bar
X\in\Gamma(TM/F_M)$ and $\bar e\in\Gamma(E/K)$ is 
\begin{equation*}
\begin{split}
&(\nabla_{Y_1}^{\operatorname{Hom}}(\omega_{\nabla}(Y_2)))(\bar X)(\bar e)-
(\nabla_{Y_2}^{\operatorname{Hom}}(\omega_{\nabla}(Y_1)))(\bar X)(\bar e)
-(\omega_{\nabla}[Y_1,Y_2])(\bar X)(\bar e)\\
=&\nabla^p_{Y_1}(\omega_{\nabla}(Y_2)(\bar X)(\bar e))-\omega_{\nabla}(Y_2)(\bar X)(\nabla^p_{Y_1}\bar e)
-\omega_{\nabla}(Y_2)(\nabla^{F_M}_{Y_1}\bar X)(\bar e)\\
&-\nabla^p_{Y_2}(\omega_{\nabla}(Y_1)(\bar X)(\bar e))+\omega_{\nabla}(Y_1)(\bar X)(\nabla^p_{Y_2}\bar e)
+\omega_{\nabla}(Y_1)(\nabla^{F_M}_{Y_2}\bar X)(\bar e)-(\omega_{\nabla}[Y_1,Y_2])(\bar X)(\bar e)\\
=&\nabla^p_{Y_1}\nabla^p_{Y_2}\bar\nabla_{X}\bar e-\cancel{\nabla^p_{Y_1}\bar\nabla_{X}\nabla^p_{Y_2}\bar e}
-\cancel{\nabla_{Y_1}^p\bar\nabla_{[Y_2,X]}\bar e}
-\cancel{\nabla^p_{Y_2}\bar\nabla_{X}\nabla^p_{Y_1}\bar e}+\bar\nabla_{X}\nabla^p_{Y_2}\nabla^p_{Y_1}\bar e
+\cancel{\bar\nabla_{[Y_2,X]}\nabla_{Y_1}^p\bar e}\\
&-\cancel{\nabla^p_{Y_2}\bar\nabla_{[Y_1,X]}\bar e}+\cancel{\bar\nabla_{[Y_1,X]}\nabla^p_{Y_2}\bar e}
+\bar\nabla_{[Y_2,[Y_1,X]]}\bar e\\
&-\nabla^p_{Y_2}\nabla^p_{Y_1}\bar\nabla_{X}\bar e+\cancel{\nabla^p_{Y_2}\bar\nabla_{X}\nabla^p_{Y_1}\bar e}
+\cancel{\nabla_{Y_2}^p\bar\nabla_{[Y_1,X]}\bar e}
+\cancel{\nabla^p_{Y_1}\bar\nabla_{X}\nabla^p_{Y_2}\bar e}-\bar\nabla_{X}\nabla^p_{Y_1}\nabla^p_{Y_2}\bar e
-\cancel{\bar\nabla_{[Y_1,X]}\nabla_{Y_2}^p\bar e}\\
&+\cancel{\nabla^p_{Y_1}\bar\nabla_{[Y_2,X]}\bar e}-\cancel{\bar\nabla_{[Y_2,X]}\nabla^p_{Y_1}\bar e}
-\bar\nabla_{[Y_1,[Y_2,X]]}\bar e\\
&-\nabla^p_{[Y_1,Y_2]}\bar\nabla_X\bar e+\bar\nabla_{X}\nabla^p_{[Y_1,Y_2]}\bar e+\bar\nabla_{[[Y_1,Y_2],X]}\bar e\\
=&R_{\nabla^p}(Y_1,Y_2)(\bar\nabla_{X}\bar e)-\bar\nabla_{X}(R_{\nabla^p}(Y_1,Y_2)\bar e)
+\bar\nabla_{[Y_2,[Y_1,X]]+[[Y_1,Y_2],X]-[Y_1,[Y_2,X]]}\bar e=0.
\end{split}
\end{equation*}
The last equality uses the flatness of $\nabla^p$ and the Jacobi
identity.  The second statement follows from the second part of Lemma
\ref{ex_of_ext}: assume that $\nabla'$ is a second extension of
$\nabla^p$, then the reader is invited to check that
\[ \omega_\nabla=\omega_{\nabla'}+\dr_{\nabla^p}\bar\phi,
\]
with
$\bar\phi\in\Gamma(\Hom(TM/F_M,\End(E/K)))=\Omega^0(F_M,\Hom(TM/F_M,\End(E/K)))$
the form induced as in Lemma \ref{ex_of_ext} by the difference
$\phi=\nabla-\nabla'$.
\end{proof}

By Proposition \ref{d_vanishes}, the cohomology class of
$\omega_\nabla$ is an invariant of the flat connection $\nabla^p$.
\begin{definition}
  Let $\nabla^p$ be a flat $F_M$-connection on $E/K$. Choose an
  extension $\nabla\colon\mx(M)\times\Gamma(E)\to\Gamma(E)$. Then the
  \emph{Atiyah class of $\nabla^p$} is the class
  \[\alpha_{\nabla^p}=[\omega_{\nabla}] \in
    H^1_{\dr_{\nabla^p}}(F_M, \operatorname{Hom}(TM/F_M,
    \operatorname{End}(E/K))).\] Alternatively, this class is also the
  \emph{Atiyah class of the infinitesimal ideal system
    $(F_M,K,\nabla^p)$ in the vector bundle $E\to M$}.
\end{definition}
The triple $(F_M,K,\nabla^p)$ defines an involutive, linear subbundle
$F_E\subseteq TE$ \cite{JoOr14}, and so a right-invariant foliation on
the principal bundle of frames of $E$. If $K=\{0\}$ then it is
horizontal and the Atiyah class above coincide with Molino's Atiyah
class of a foliated bundle \cite{Molino71b}.

The definition of the Atiyah class of a flat partial connection is
motivated by the following theorem, the proof of which is now easy to
complete using Proposition \ref{existence_of_connection} and Lemma \ref{cond_for_proj}.
\begin{theorem}
  Assume that $\varphi\colon E\to E'$ is a fibration of vector bundles
  over $\varphi_0\colon M\to M'$. Consider the flat connection
  $\nabla^\varphi$ defined by the fibration as in Proposition \ref{prop2.2}.
  Then $\alpha_{\nabla^\varphi}$ vanishes.
\end{theorem}
This yields immediately the following obstructions to a flat
$F_M$-connection on $E/K$ defining a fibration.
\begin{corollary}\label{quotient_cor}
  Let $E\to M$ be a smooth vector bundle, $K\subseteq E$ a subbundle
  and $F_M\subseteq TM$ an involutive subbundle.  Let $\nabla^p\colon
  \Gamma(F_M)\times\Gamma(E/K)\to\Gamma(E/K)$ be a flat connection.
  If $\nabla^p$ induces a fibration of vector bundles $E\to
  (E/K)/\nabla^p$ over $M\to M/F_M$, then $\alpha_{\nabla^p}=0$.
\end{corollary}

\begin{corollary}\label{atiyah_simple_fol}\cite{Molino71b}
  Let $M$ be a smooth manifold and $F_M$ an involutive subbundle of
  $TM$. If $F_M$ is simple, % and $\nabla^{F_M}$ has no holonomy
  then the Atiyah class of $\nabla^{F_M}$ vanishes.
\end{corollary}

\subsection{Infinitesimal ideal systems}\label{prelim_iis}
This subsection discusses the notion of infinitesimal ideal system in
a general Lie algebroid, and defines the Atiyah class of such an
ideal.
\subsubsection{Main properties of infinitesimal ideal systems}
Consider an ideal $(F_M,J,\nabla)$ in a Lie algebroid $A\to M$. Since
for $j\in \Gamma(J)$, the class $\bar j=0\in\Gamma(A/J)$, (iis1) implies
that $[j,j']\in\Gamma(J)$ for all $j'\in\Gamma(J)$, i.e.~$J$ is
automatically a Lie subalgebroid of $A$.  The following
proposition reformulates the ideal condition (iis1) in the
definition of an infinitesimal ideal system.
\begin{proposition}\label{eq_iis}
Let $A\to M$ be a Lie algebroid, $F_M\subseteq TM$ an involutive
subbundle and $J\subseteq A$ a subalgebroid with $\rho(J)\subseteq
F_M$. Let $\nabla\colon \Gamma(F_M)\times\Gamma(A/J)\to\Gamma(A/J)$ be a
flat connection. Then the following are equivalent.
\begin{enumerate}
\item[(iis1)] $[a,j]\in\Gamma(J)$ 
for all $j\in\Gamma(J)$ and $a\in\Gamma(A)$ $\nabla$-parallel;
\item[(iis1')] $\nabla^J_j\bar a=\nabla_{\rho(j)}\bar a$ for all $j\in
  \Gamma(J)$ and $a\in\Gamma(A)$.
\end{enumerate}
\end{proposition}

\begin{proof}
First assume that  $[a,j]\in\Gamma(J)$ 
for all $j\in\Gamma(J)$ and $a\in\Gamma(A)$ $\nabla$-parallel. Then
for all $\nabla$-parallel sections $\bar a$ of $A/J$:
\[\nabla^J_j\bar a=\overline{[j,a]}=0\]
for all $j\in\Gamma(J)$. That is, a $\nabla$-parallel section of $A$
is also $\nabla^J$-parallel. Since $\Gamma(A/J)$ has local basis
frames consisting of $\nabla$-parallel sections of $A/J$, see
\cite{JoOr14}, conclude as follows. Take $a\in\Gamma(A)$ defined on a
neighbourhood of $p\in M$. Then there is an open set $U\subseteq M$,
$p\in U$, smooth functions $f_1,\ldots,f_r\in C^\infty(U)$ and
$\bar a_1,\ldots, \bar a_r\in\Gamma(A/J)^\nabla$ such that
$\bar a=\sum_{i=1}f_i\bar a_i$. Compute then for $j\in \Gamma(J)$:
\[\nabla^J_j\bar a=\sum_{i=1}\ldr{\rho(j)}(f_i)\bar
a_i=\nabla_{\rho(j)}\bar a.\]

Conversely, assume that $\nabla^J_j\bar a=\nabla_{\rho(j)}\bar a$ for
all $j\in \Gamma(J)$ and $a\in\Gamma(A)$. Then for a $\nabla$-flat
section $a\in\Gamma(A)$ the class $\overline{[j,a]}$ equals
$\nabla^J_j\bar a=\nabla_{\rho(j)}\bar a=0$ and so $[j,a]\in\Gamma(J)$
for all $j\in\Gamma(J)$.
\end{proof}

This implies a simplification of the definition of an ideal in the
case $F_M=\rho(J)$.
\begin{proposition}\label{simpler_def_iis}
Let $A\to M$ be a Lie algebroid
Let $A\to M$ be a Lie algebroid, $F_M\subseteq TM$ an involutive
subbundle and $J\subseteq A$ a subbundle with $\rho(J)=
F_M$. Let $\nabla\colon \Gamma(F_M)\times\Gamma(A/J)\to\Gamma(A/J)$ be a
flat connection. 
Then (iis2) and (iis3) in the definition of an ideal follow from (iis1).
\end{proposition}

\begin{proof}
First take two $\nabla$-flat sections $a,b\in\Gamma(A)$ and take any
$X\in\Gamma(F_M)$. Then there exists $j\in\Gamma(J)$ with $\rho(j)=X$
and compute 
\[\nabla_X\overline{[a,b]}=\nabla_{\rho(j)}\overline{[a,b]}=\nabla^J_j
\overline{[a,b]}=\overline{[j,[a,b]]}
=\overline{[[j,a],b]}+\overline{[a,[j,b]]}.
\]
Since $a,b$ are $\nabla$-flat, the brackets $[j,a], [j,b]\in\Gamma(J)$, and
again since $a,b$ are $\nabla$-flat, 
$[[j,a],b]+[a,[j,b]]\in\Gamma(J)$. Therefore
$\nabla_X\overline{[a,b]}=0$,
which proves (iis2).

In the same manner, take $a\in\Gamma(A)$ a $\nabla$-flat section and
take $X\in\Gamma(F_M)$. Then there exists $j\in\Gamma(J)$ such that
$\rho(j)=X$
and  compute 
$\nabla_X^{F_M}\overline{\rho(a)}=\overline{[X,\rho(a)]}=\overline{\rho[j,a]}$.
Then $[j,a]\in\Gamma(J)$ implies $\rho[j,a]\in\Gamma(F_M)$ and so $\nabla_X^{F_M}\overline{\rho(a)}=0$.
\end{proof}

\subsubsection{Fibrations of Lie algebroids, infinitesimal ideal systems and the Atiyah class}
The following theorem shows that infinitesimal ideal systems in Lie
algebroids define quotients of Lie algebroids, up to some topological
obstructions. The paper \cite{DrJoOr15} proves that an infinitesimal
ideal system defines a sub-representation (up to homotopy) of the
adjoint representation of the Lie algebroid, after the choice of an
extension of the infinitesimal ideal system connection.  These two
results suggest that indeed, an infinitesimal ideal systems is the
right notion of ideal in a Lie algebroid. 
\begin{theorem}\label{red_lie_alg}\cite{JoOr14}
  Let $(F_M,J,\nabla)$ be an ideal in a Lie algebroid $A$. Assume that
  $\bar M=M/F_M$ is a smooth manifold and that $\nabla$ has trivial
  holonomy.  Then the vector bundle $(A/J)/\nabla\to M/F_M$ carries a
  Lie algebroid structure such that the projection $(\pi,\pi_M)$ is a
  Lie algebroid morphism.
\begin{align*}
\begin{xy}
  \xymatrix{
    A\ar[d]_{q_A}\ar[r]^{\pi}&(A/J)/\nabla\ar[d]^{[q_A]}\\
    M\ar[r]_{\pi_M}&M/F_M }
\end{xy}
\end{align*}
\end{theorem}
Conversely, let \begin{equation*}
\begin{xy}
\xymatrix{
A\ar[r]^{\varphi}\ar[d]_{q_A}& A'\ar[d]^{q_{A'}}\\
M\ar[r]_f&M'
}
\end{xy}
\end{equation*}
be a fibration of Lie algebroids, i.e.~a Lie algebroid morphism such that
$(\varphi,\varphi_0)$ is a fibration of vector bundles.
Then $J:=\ker(\varphi)\subseteq A$ is a subalgebroid of $A$ and
$F_M=T^{\varphi_0}M\subseteq TM$ is an involutive subbundle.  The equality
$T\varphi_0\circ\rho=\rho'\circ\varphi$ yields immediately $\rho(J)\subseteq
F_M$.

Define as before the connection
$\nabla^\varphi\colon \Gamma(F_M)\times \Gamma(A/J)\to\Gamma(A/J)$ by
setting $\nabla^\varphi_X\bar a=0$ for all sections $a\in \Gamma(A)$
that are $\varphi$-related to some section $a'\in\Gamma(A')$.  Then
the properties of the Lie algebroid morphism $(\varphi, \varphi_0)$
imply that $(F_M, J, \nabla^\varphi)$ is an infinitesimal ideal system
in $A$, see \cite{JoOr14}.

Note that an ideal $(F_M,J,\nabla)$ is defined as above by the kernel
of a fibration of Lie algebroids if and only if it integrates to an
ideal system in the sense of Higgins and Mackenzie \cite{HiMa90b,
  Mackenzie05}, see \cite{JoOr14}.

\begin{example}[Flat connections on vector bundles]\label{geom_red}
  Let $E$ be a vector bundle over $M$ and $K$ a subbundle. Then any
  flat connection of an involutive subbundle $F_M\subseteq TM$ on
  $E/K$ as in \S \ref{sec_fibrations} defines an infinitesimal ideal
  system $(F_M, K, \nabla)$ in the Lie algebroid
  $(E, \rho=0, [\cdot\,,\cdot]=0)$.  Here, Theorem \ref{main} is
  Corollary \ref{quotient_cor}.

  An infinitesimal ideal system in a Lie algebroid is, by forgetting
  the Lie algebroid structure, automatically an infinitesimal ideal
  system in the underlying vector bundle.  The Atiyah class of a
  general infinitesimal ideal system is defined below as the Atiyah
  class of the infinitesimal ideal system in the underlying vector
  bundle.
\end{example}

\begin{definition}
  The \emph{Atiyah class of an infinitesimal ideal system}
  $(F_M,J,\nabla)$ in a Lie algebroid $A\to M$ is the Atiyah
  class of the flat connection $\nabla\colon\Gamma(F_M)\times\Gamma(A/J)\to\Gamma(A/J)$.
\end{definition}

By the considerations above, the Atiyah class of an infinitesimal
ideal system is really just the Atiyah class of the connection, and
defines an obstruction to the ideal defining a fibration of vector
bundles. Indeed, a review of the proof of Theorem \ref{red_lie_alg} in
\cite{JoOr14} reveals that if the fibration of vector bundles is
well-defined, then the Lie algebroid structure on the quotient comes
`for free' along. The following result follows hence immediately from
Corollary \ref{quotient_cor}.
\begin{proposition}\label{atiyah0}
  Let $(F_M,J,\nabla)$ be an ideal in a Lie algebroid $A\to M$.  If
  the quotient Lie algebroid $(A/J)/\nabla\to M/F_M$ as in Theorem
  \ref{red_lie_alg} exists, then the Atiyah class 
  \[\alpha_\nabla\in H^1(F_M,\Hom(TM/F_M,\End(A/J)))\]
  of the ideal vanishes.
  \end{proposition}

  The remainder of this section describes two natural examples of ideals.

\begin{example}[The usual naive notion of ideal in a Lie algebroid]\label{naive_ideal}
  A (naive) ideal $I$ in a Lie algebroid $A\to M$ is a subbundle over
  $M$ such that $[a,i]\in\Gamma(I)$ for all $i\in\Gamma(I)$ and all
  $a\in\Gamma(A)$.  The inclusion $I\subseteq\ker(\rho)$ follows
  immediately and shows that this definition of an ideal is very
  restrictive. These usual ideals correspond obviously to the
  infinitesimal ideal systems $(F_M=0,J=I,\nabla=0)$ in $A$. In
  particular, an ideal in a Lie algebra is an infinitesimal ideal
  system. In this case, the quotient Lie algebroid $A/I$ over
  $M/F_M=M$ is always defined.

  It is easy to check that the Atiyah class of the Lie pair $(A,I)$ is
  zero, which coincides with the fact that the quotient Lie algebroid
  $A/I\to M$ is always defined.
\end{example}

\begin{example}[The Bott connection and reduction by simple  foliations]\label{bott_ideal}
  A standard example of a Lie algebroid is the tangent space $TM$ of a
  smooth manifold $M$, endowed with the usual Lie bracket of vector
  fields and the identity $\id_{TM}$ as anchor.  Consider an
  involutive subbundle $F_M\subseteq TM$ and the Bott connection
\[\nabla^{F_M}\colon\Gamma(F_M)\times\Gamma(TM/F_M)\to\Gamma(TM/F_M)\]
associated to it. Then Propositions \ref{eq_iis} and
\ref{simpler_def_iis} show that $(F_M,F_M,\nabla^{F_M})$ is an ideal
in $TM$.  If the quotient Lie algebroid exists, then it is isomorphic
to $T(M/F_M)\to M/F_M$, see \cite{JoOr14}.
For this class of infinitesimal ideal system, Theorem \ref{main} is
Corollary \ref{atiyah_simple_fol}.
\end{example}

\section{Geometric interpretation of the Atiyah class of a Lie pair}\label{sec_proof_main}

This section proves Theorem \ref{main}. The first subsection recalls
the construction of the Atiyah class of a Lie pair $(A,J)$
\cite{ChStXu16}. The second subsection explains how it can be
constructed from an extension of an infinitesimal ideal system
structure with fiber $J$, and proves the theorem.

\subsection{The Atiyah class of a Lie pair}
Consider a Lie algebroid $A$ over a smooth manifold $M$, together with
a subalgebroid $J\subseteq A$.  The pair $(A,J)$ is a \emph{Lie pair}
\cite{ChStXu16}. Recall that the Lie pair defines the flat \emph{Bott
  connection}
\[\nabla^J\colon\Gamma(J)\times\Gamma(A/J)\to\Gamma(A/J), \quad \nabla^J_j\bar a=\overline{[a,j]}
\]
for $j\in\Gamma(J)$, and $\bar a\in\Gamma(A/J)$ the class of
$a\in\Gamma(A)$. This induces as usual a flat $J$-connection
$\nabla^{\operatorname{Hom}}$ on
$\operatorname{Hom}(A/J,\operatorname{End}(A/J))=(A/J)^*\otimes
(A/J)^*\otimes (A/J)$:
\[(\nabla^{\operatorname{Hom}}_j\phi)(\bar a_1,\bar a_2)=\nabla^J_j(\phi(\bar a_1,\bar a_2))
-\phi(\nabla_j^{J}\bar a_1,\bar a_2)-\phi(\bar a_1,\nabla_j^{J}\bar a_2).
\] 
For simplicity, $\dr_{\nabla^J}$ denotes the Koszul differential on
$\Omega^\bullet(J, (A/J)^*\otimes(A/J)^*\otimes(A/J))$ defined by this
connection.

Choose an extension
$\nabla^A\colon\Gamma(A)\times\Gamma(A)\to\Gamma(A)$ of
$\nabla^J$. That is, $\nabla^A_aj\in\Gamma(J)$ for all $a\in\Gamma(A)$
and $j\in\Gamma(J)$ and the induced quotient connection
$\overline{\nabla^A}\colon\Gamma(A)\times\Gamma(A/J)\to\Gamma(A/J)$
restricts to $\nabla^J$ on sections of $J$ in the first argument.
Then $\omega_{\nabla^A}\in\Omega^1(J,\Hom(A/J,\End(A/J)))$ defined by
\[\omega_{\nabla^A}(j)(\bar a)(\bar b)=\overline{R_{\nabla^A}(j,a)b},
  \quad j\in\Gamma(J), a,b\in\Gamma(A),
\]
satisfies $\dr_{\nabla^J}\omega_{\nabla^A}=0$ -- this works as in
Proposition \ref{d_vanishes}, see \cite{ChStXu16}.  The \emph{Atiyah
  class of the Lie pair} $(A,J)$ is the cohomology class
\[\alpha_{J}:=[\omega_{\nabla^A}]\in
  H^1_{\dr_{\nabla^J}}(J,
  \operatorname{Hom}(A/J,\operatorname{End}(A/J))),\] see
\cite{ChStXu16}.  It does not depend on the choice of the extension
$\nabla^A$ of the Bott connection $\nabla^J$.

\subsection{Proof of Theorem \ref{main}}
Consider now an infinitesimal ideal system $(F_M,J,\nabla^i)$ in
$A$. Choose an extension
$\nabla\colon\mx(M)\times\Gamma(A)\to\Gamma(A)$ of $\nabla^i$. That
is, $\nabla$ preserves sections of $J$ and quotients to a connection
$\bar\nabla\colon\mx(M)\times\Gamma(A/J)\to\Gamma(A/J)$ that restricts
to $\nabla^i$ on sections of $F_M$.

Consider the basic connection
$\nabla^{\rm bas}\colon\Gamma(A)\times\Gamma(A)\to\Gamma(A)$ induced
by $\nabla$. It is defined by
$\nabla^{\rm bas}_{a_1}a_2=[a_1,a_2]+\nabla_{\rho(a_2)}a_1$. Then a
section $a\in\Gamma(A)$ such that $\bar a\in\Gamma(A/J)$ is
$\nabla^i$-flat satisfies
\begin{equation}
\nabla^{\rm bas}_aj=[a,j]+\nabla_{\rho(j)}a\in\Gamma(J)
\end{equation}
for all $j\in\Gamma(J)$. To see this, recall that $[a,j]\in\Gamma(J)$
by the definition of an infinitesimal ideal system, and
$\rho(j)\in\Gamma(F_M)$ induces
$\overline{\nabla_{\rho(j)}a}=\bar\nabla_{\rho(j)}\bar
a=\nabla^i_{\rho(j)}\bar a=0$ in $\Gamma(A/J)$. Since $\nabla^i$-flat
sections generate all sections of $A$, one finds that $\nabla^{\rm
  bas}_aj\in\Gamma(J)$ for all $j\in\Gamma(J)$ and \emph{all}
$a\in\Gamma(A)$.
This proves the following lemma.
\begin{lemma}\label{compatibility_extensions_1}
  Consider a Lie algebroid $A$ over a smooth manifold $M$, and an
  infinitesimal ideal system $(F_M,J,\nabla^i)$ in $A$. Choose an
  extension $\nabla\colon\mx(M)\times\Gamma(A)\to\Gamma(A)$ of
  $\nabla^i$. Then the basic connection
  $\nabla^{\rm bas}\colon \Gamma(A)\times\Gamma(A)\to\Gamma(A)$
  defined by $\nabla$ satisfies $\nabla^{\rm bas} _a j\in\Gamma(J)$
  for all $a\in\Gamma(A)$ and all $j\in\Gamma(J)$.
  \end{lemma}
As a consequence, there is a quotient connection
\[\overline{\nabla^{\rm bas}}\colon \Gamma(A)\times\Gamma(A/J)\to\Gamma(A/J), \quad 
  \overline{\nabla^{\rm bas}}_{a_1}\bar{a_2}=\overline{\nabla^{\rm
      bas}_{a_1}a_2}.\]

The following proposition shows that
$\nabla^{\rm bas}$ is an extension of the Bott connection
$\nabla^J\colon\Gamma(J)\times\Gamma(A/J)\to\Gamma(A/J)$.
\begin{proposition}\label{compatibility_extensions_2}
  Consider a Lie algebroid $A$ over a smooth manifold $M$, and an
  infinitesimal ideal system $(F_M,J,\nabla^i)$ in $A$. Choose an
  extension $\nabla\colon\mx(M)\times\Gamma(A)\to\Gamma(A)$ of
  $\nabla^i$.  Then the quotient connection
  $\overline{\nabla^{\rm bas}}$ restricts on sections of $J$ to the
  Bott connection
  $\nabla^J\colon \Gamma(J)\times\Gamma(A/J)\to\Gamma(A/J)$,
  $\nabla^J_j\bar a =\overline{[j,a]}\in \Gamma(A/J)$.
\end{proposition}

\begin{proof}
Choose  $j\in\Gamma(J)$ and $a\in\Gamma(A)$. Then
\begin{equation}
  \overline{\nabla^{\rm bas}}_{j}\bar{a}=\overline{\nabla^{\rm bas}_{j}a}
=\overline{[j,a]+\nabla_{\rho(a)}j}=\overline{[j,a]}=\nabla^J_j\bar a.
\end{equation}
In the third equation, $\nabla_Xj\in\Gamma(J)$ for all $X\in\mx(M)$
and $j\in\Gamma(J)$ is given by the choice of the connection $\nabla$.
\end{proof}

Hence, $\nabla^{\rm bas}$ can be used to construct the Atiyah class $\alpha_J$ of the Lie pair $(A,J)$:
\[\alpha_J=\left[\omega_{\nabla^{\rm bas}}\right]\in H^1(J, \Hom(A/J,\End(A/J)))\]
for $\nabla^{\rm bas}$ defined by an extension $\nabla$ of the
infinitesimal ideal system connection $\nabla^i$.

\medskip

The remainder of this section proves Theorem \ref{main}. The key to
this proof is the following: an infinitesimal ideal system
$(F_M,J,\nabla^i)$ in $A$ makes the anchor $\rho$ into a chain
map
\[\rho^\star\colon \left(\Omega^\bullet(F_M,\Hom(TM/F_M, \End(A/J))),
    \dr_{\nabla^i}\right)\to
  \left(\Omega^\bullet(J,\Hom(A/J, \End(A/J))), \dr_{\nabla^J}\right).
  \]
\begin{lemma}
Let $(J,A)$ be a Lie pair over a smooth manifold $M$. Let
$F_M\subseteq TM$ be an involutive subbundle with $\rho(J)\subseteq F_M$ and
let 
$\nabla\colon \Gamma(F_M)\times\Gamma(A/J)\to\Gamma(A/J)$
be a linear connection.
Then  \begin{equation*}
  \begin{split}\rho^\star&\colon \Omega^\bullet(F_M, \operatorname{Hom}(TM/F_M,\operatorname{End}(A/J)))
\to \Omega^\bullet(J,\operatorname{Hom}(A/J,\operatorname{End}(A/J))),\\
&(\rho^\star\omega)(j_1, \ldots, j_l)(\overline{a_1},
\overline{a_2})=\omega(\rho(j_1), \ldots,
\rho(j_l))(\overline{\rho(a_2)})(\overline{a_1})
\end{split}
\end{equation*}
is a well-defined degree-preserving morphism of $C^\infty(M)$-modules
and a morphism of modules over
$\rho^\star\colon \Omega^\bullet(F_M)\to\Omega^\bullet(J)$.
\end{lemma}
\begin{proof}
  The map $\rho^\star$ is well-defined because $\rho(J)\subseteq F_M$, and $\bar a=0\in A/J$ if
  and only if $a\in\Gamma(J)$, which implies $\rho(a)\in\Gamma(F_M)$
  and so $\overline{\rho(a)}=0$ in $TM/F_M$.
\end{proof}

\begin{lemma}\label{lem2}
  Let $A\to M$ be a Lie algebroid and let $(F_M,J,\nabla)$ be an
  infinitesimal ideal system in $A$.  Then $\rho^\star$ satisfies
  $\dr_{\nabla^J}\circ\rho^\star=\rho^\star\circ\dr_{\nabla^i}$.
\end{lemma}

\begin{remark}
  Note that conversely,
  $\dr_{\nabla^J}\circ\rho^\star=\rho^\star\circ\dr_{\nabla^i}$ does not
  imply that $(F_M,J,\nabla)$ is an infinitesimal ideal system.  In
  general, the following result holds.  Let
  $\phi\in\Omega^1(J,\operatorname{End}(A/J))$ be defined by
  $\phi(j)\bar a=\nabla_{\rho(j)}\bar a-\nabla^J_j\bar a$ for all
  $j\in\Gamma(J)$ and all $a\in\Gamma(A)$.  Consider the open subset
  $U\subseteq M$ defined by
  $U:=\{p\in M\mid \rho(A(p))\not\subseteq F_M(p)\}$.  Then
  $\rho^\star$ satisfies
  $\dr_{\nabla^J}\circ\rho^\star=\rho^\star\circ\dr_{\nabla^i}$ if and
  only if there exists $\lambda\in\Gamma_U(J^*)$ such that
  $\phi\an{U}=\lambda\cdot\operatorname{id}_{A/J}$.
\end{remark}

\begin{proof}[Proof of Lemma \ref{lem2}]
  Choose
  $\omega\in\Omega^l_U(F_M,
  \operatorname{Hom}(TM/F_M,\operatorname{End}(A/J)))$. Then on the
  one hand,
\begin{equation}\label{eq1}
\begin{split}
  &\rho^\star(\dr_{\nabla^i}\omega)(j_1, \ldots,
  j_{l+1})(\overline{a_1}, \overline{a_2})
  =(\dr_{\nabla^i}\omega)(\rho(j_1), \ldots, \rho(j_{l+1}))(\overline{\rho(a_2)})(\overline{a_1})\\
  &=\sum_{i=1}^{l+1}(-1)^{i+1}\nabla^{\operatorname{Hom}}_{\rho(j_i)}(\omega(\rho(j_1),
  \ldots,\hat{i}, \ldots,
  \rho(j_{l+1})))(\overline{\rho(a_2)})(\overline{a_1})\\
  &\quad+\sum_{i<l}(-1)^{i+l}\omega([\rho(j_i),\rho(j_l)],\rho(j_1),
  \ldots,\hat{i}, \ldots, \hat{l}, \ldots
  \rho(j_{l+1}))(\overline{\rho(a_2)})(\overline{a_1}).
\end{split}
\end{equation}
On the other hand 
\begin{equation}\label{eq2}
\begin{split}
  &\dr_{\nabla^J}(\rho^\star\omega)(j_1, \ldots, j_{l+1})(\overline{a_1}, \overline{a_2})\\
  &=\sum_{i=1}^{l+1}(-1)^{i+1}\nabla^{\operatorname{Hom}}_{j_i}(\rho^\star\omega(j_1,
  \ldots,\hat{i}, \ldots, j_{l+1}))(\overline{a_1}, \overline{a_2})\\
  &\quad+\sum_{i<l}(-1)^{i+l}\rho^\star\omega([j_i,j_l], j_1,
  \ldots,\hat{i}, \ldots, \hat{l}, \ldots
  j_{l+1})(\overline{a_1}, \overline{a_2})\\
  &=\sum_{i=1}^{l+1}(-1)^{i+1}\nabla^{\operatorname{Hom}}_{j_i}(\rho^\star\omega(j_1,
  \ldots,\hat{i}, \ldots, j_{l+1}))(\overline{a_1}, \overline{a_2})\\
  &\quad+\sum_{i<l}(-1)^{i+l}\omega(\rho[j_i,j_l], \rho(j_1),
  \ldots,\hat{i}, \ldots, \hat{l}, \ldots
  \rho(j_{l+1}))(\overline{\rho(a_2)})(\overline{a_1}).
\end{split}
\end{equation}
Proposition \ref{eq_iis} and the compatibility of the Lie bracket with the
anchor give for all $j, j_1, \ldots, j_l\in\Gamma(J)$ and
$a_1, a_2\in\Gamma(A)$:
\begin{equation*}
\begin{split}
  &\nabla^{\operatorname{Hom}}_{j}(\rho^\star\omega(j_1,
  \ldots, j_{l}))(\overline{a_1})(\overline{a_2})\\
  &=
  \nabla^J_j\left(\omega(\rho(j_1),\ldots, \rho(j_{l}))(\overline{\rho(a_2)})(\overline{a_1})\right)-\omega(\rho(j_1),\ldots, \rho(j_{l}))(\overline{\rho(a_2)})(\nabla_j^J\overline{a_1})\\
  &\quad-\omega(\rho(j_1),\ldots, \rho(j_{l}))(\overline{\rho([j,a_2])})(\overline{a_1})\\
  &= \nabla^i_{\rho(j)}\left(\omega(\rho(j_1),\ldots,
    \rho(j_{l}))(\overline{\rho(a_2)})(\overline{a_1})\right)-\omega(\rho(j_1),\ldots, \rho(j_{l}))(\overline{\rho(a_2)})(\nabla_{\rho(j)}^i\overline{a_1})\\
  &-\omega(\rho(j_1),\ldots, \rho(j_{l}))(\nabla^{F_M}_{\rho(j)}\overline{\rho(a_2)})(\overline{a_1})\\
  &=\nabla^{\operatorname{Hom}}_{\rho(j)}(\omega(\rho(j_1), \ldots,
  \rho(j_{l})))(\overline{\rho(a_2)})(\overline{a_1}),
\end{split}
\end{equation*}
Hence, by \eqref{eq1} and \eqref{eq2},
$\rho^\star\circ\dr_{\nabla^i}=\dr_{\nabla^J}\circ\rho^\star$ and the
proof is complete.
\end{proof}

As a consequence, if the Lie pair $(A,J)$ carries an ideal structure
$(F_M,J,\nabla^i)$, then $\rho^\star$ induces a map in cohomology. The
following theorem shows that the Atiyah class of the Lie pair is then
the image under this map of the Atiyah class of the infinitesimal
ideal system.
\begin{theorem}\label{last_key}
  Let $A\to M$ be a Lie algebroid and let $(F_M,J,\nabla^i)$ be an ideal in $A$.
The image under $\rho^\star$ of the Atiyah class 
\[\alpha_{\nabla^i}\in H^1_{\dr_{\nabla^i}}(F_M, \operatorname{Hom}(TM/F_M,\operatorname{End}(A/J)))
\]
of the ideal is the Atiyah class of the Lie pair $(A,J)$
\[\alpha_{J}\in H^1_{\dr_{\nabla^J}}(J, \operatorname{Hom}(A/J,\operatorname{End}(A/J))).
\]
\end{theorem}

\begin{proof}
  Take an extension $\nabla$ of $\nabla^i$, and consider the
  associated basic connection $\nabla^{\rm bas}$. By Proposition
  \ref{compatibility_extensions_2}, it is an extension of the
  Bott connection $\nabla^J$ and so the Atiyah class $\alpha_J$ is the
  cohomology class of $\omega_{\nabla^{\rm bas}}$.  It suffices
  therefore to show that 
  $\rho^\star\omega_{\nabla}=\omega_{\nabla^{\rm bas}}$. This is a
  simple computation:
\begin{equation*}
\begin{split}
  \omega_{\nabla^{\rm
        bas}}(j,\overline{a_1})(\overline{a_2})
  &=\nabla_j^J\overline{\nabla^{\rm
      bas}}_{a_1}\overline{a_2}-\overline{\nabla^{\rm
      bas}}_{a_1}\nabla_j^J\overline{a_2}
  -\overline{\nabla^{\rm bas}}_{[j,a_1]}(\overline{a_2})\\
  &=\nabla_j^J\overline{[a_1,a_2]+\nabla_{\rho(a_2)}a_1}-\overline{\nabla^{\rm
      bas}}_{a_1}\overline{[j,a_2]}
  -\overline{[[j,a_1],a_2]+\nabla_{\rho(a_2)}[j,a_1]}\\
  &=\overline{[j,[a_1,a_2]+\nabla_{\rho(a_2)}a_1]}-\overline{[a_1,[j,a_2]]+\nabla_{\rho[j,a_2]}a_1}-\overline{[[j,a_1],a_2]+\nabla_{\rho(a_2)}[j,a_1]}\\
&=\overline{[j,\nabla_{\rho(a_2)}a_1]}-\overline{\nabla_{\rho[j,a_2]}a_1}-\overline{\nabla_{\rho(a_2)}[j,a_1]}\\
&=\nabla^J_j\bar\nabla_{\rho(a_2)}\overline{a_1}-\bar\nabla_{\rho[j,a_2]}\overline{a_1}
-\bar\nabla_{\rho(a_2)}\nabla^J_j\overline{a_1}\\
&=\nabla^i_{\rho(j)}\bar\nabla_{\rho(a_2)}\overline{a_1}-\bar\nabla_{[\rho(j),\rho(a_2)]}\overline{a_1}
-\bar\nabla_{\rho(a_2)}\nabla^i_{\rho(j)}\overline{a_1}=\omega_{\nabla}(\rho(j),\overline{\rho(a_2)})(\overline{a_1})\\
&=\rho^\star\omega_{\nabla}(j,\overline{a_1})(\overline{a_2})
\end{split}
\end{equation*}
for $j\in\Gamma(J)$ and $a_1,a_2\in\Gamma(A)$.
\end{proof}
Theorem \ref{last_key} and Proposition \ref{atiyah0} now induce
Theorem \ref{main}.

\def\cprime{$'$} \def\polhk#1{\setbox0=\hbox{#1}{\ooalign{\hidewidth
  \lower1.5ex\hbox{`}\hidewidth\crcr\unhbox0}}} \def\cprime{$'$}
  \def\cprime{$'$}
\providecommand{\bysame}{\leavevmode\hbox to3em{\hrulefill}\thinspace}
\providecommand{\MR}{\relax\ifhmode\unskip\space\fi MR }
% \MRhref is called by the amsart/book/proc definition of \MR.
\providecommand{\MRhref}[2]{%
  \href{http://www.ams.org/mathscinet-getitem?mr=#1}{#2}
}
\providecommand{\href}[2]{#2}

\end{document}